\documentclass[12pt]{article}

\usepackage{amsmath}   
\usepackage{amsthm}

\newtheorem{theorem}{Theorem}

\newtheorem*{lemma*}{Lemma}
\newtheorem*{remark*}{Remark}
\newtheorem{exercise}{Exercise}

\usepackage{graphicx}
\usepackage[nooneline]{caption} 
\RequirePackage{palatino}

\begin{document}


\begin{center}
{\Large\bf Some remarks on the direct calculation \\ of Probabilities in Urn Schemes}
\end{center}
\vspace{.1cm}
\begin{center}
{\sc Azam A.~Imomov, Yorqin Khodjaev}\\
\vspace{.4cm}
{\small \it Karshi State University, 17 Kuchabag street, \\
100180 Karshi city, Uzbekistan\\}
{e-mail: {\sl imomov{\_}\,azam@mail.ru}\\
\vspace{.4cm}
{\textit{\small \textbf{Dedicated to our Parents}}}}
\end{center}

\vspace{.3cm}
\begin{abstract}
    The paper considers urn schemes in which several urns can be involved. Simplified formulas
    are proposed that allow direct calculation of probabilities without the use of elements combinatorics.


\end{abstract}


\section*{Introduction}

    The urn schemes is one of the simplest models used in the elementary probability theory. Using
    urn schemes, it is convenient to calculate some basic probabilities through conditional probabilities.
    In most cases, when solving various problems, a model with a single urn is considered. For various
    values of the parameters of the scheme, many well-known schemes of probability theory are obtained,
    in particular, a random choice scheme with return as Bernoulli tests, a random choice without return 
    scheme, Ehrenfest diffusion model {\cite{Bingham}}, and P\'{o}lya urn models. These schemes serve as models
    of many real phenomena, as well as methods for their investigation; see, for example, {\cite{PUrn}}, {\cite{Johnson}}. 

    One of the simplest models of urn schemes is the following model. Consider two urns, which we
    denote by the symbols $\Pi _0$ and $\Pi$ respectively. In both urns there are a lot of white and
    black balls. The following operation is allowed. The random number $\xi$ of balls are taken out
    from the urn $\Pi _0$ in a random order, and they are transferred to the urn $\Pi$.

    {\textbf{Scheme} $\left[\textbf{A}\right]$}. Let urn $\Pi _0$ contains $a$ white and
    $b$ black balls, and there are $c$ white and $d$ black balls in the urn $\Pi$. We consider
    special case $\xi =1$: one ball is taken from urn $\Pi _0$ at random and transferred to
    urn $\Pi$. Henceforth we denote by $A$ the event that the taken out ball to be white.
    We illustrate Scheme $\left[\textbf{A}\right]$ in Figure~1.

\begin{figure}[htbp]
\begin{center}
  \includegraphics[width=3.2in]{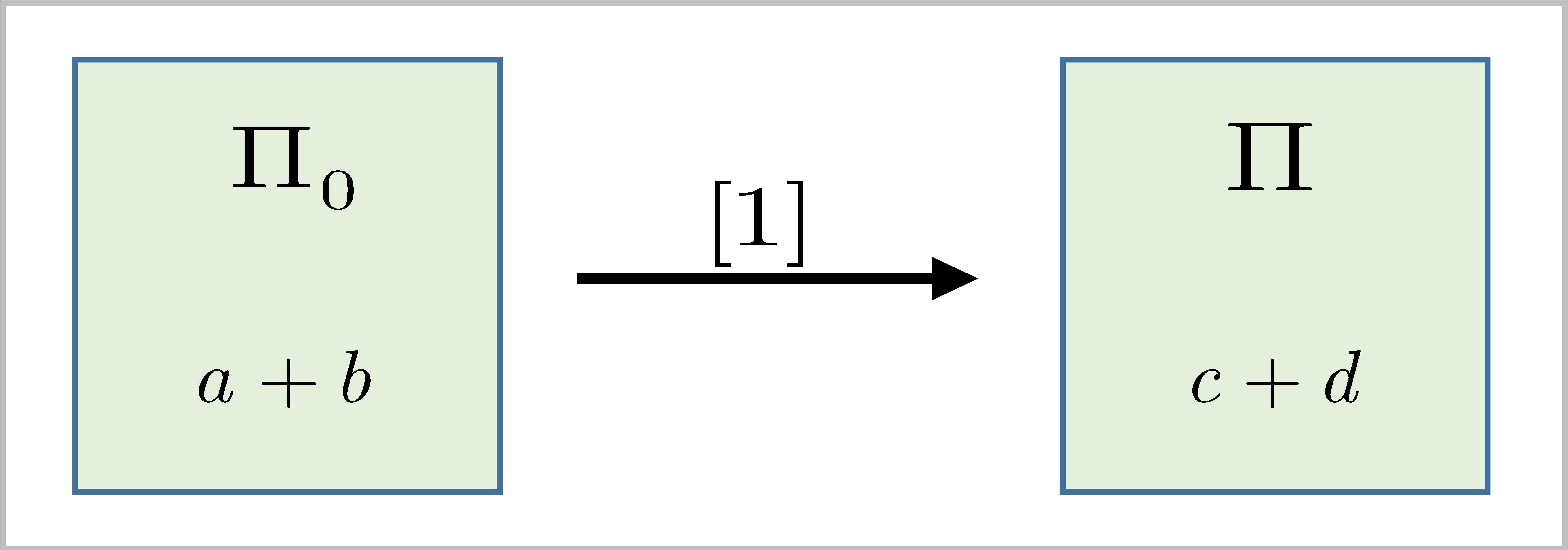}\\
  \caption{\textit{Scheme} $\left[\textbf{A}\right]$} \label{Figure_1}
  \end{center}
\end{figure}

    The probability
$$
    \textsf{P}_\Pi \left({\left. A \right|\xi =1}\right)
$$
    of that the ball randomly taken out from urn $\Pi$, under condition $\xi =1$, to be white,
    according to the classical scheme, is calculated by the following formula:
\begin{eqnarray}           \label{1}
    {\textsf{P}_\Pi \left({\left. A \right|\xi =1}\right)} \nonumber
    & = & {1 \over {\textsf{C}_{a+b}^1}}\left( {{{c+1}  \over
    {c+d+1}} \textsf{C}_a^1+ {c \over {c+d+1}} \textsf{C}_b^1}\right) \\       \nonumber
\\
    & = & {{a+(a+b)c} \over {(a+b)(c+d+1)}}={c \over {c+d+1}}+{a \over {a+b}}{1 \over {c+d+1}} \raise 1.8pt\hbox{.}
\end{eqnarray}
    where $\textsf{C}_n^m  = {{n!} \over {m!(n - m)!}}$ -- Newton's binomial coefficients. In fact, this
    is a well-known elementary formula for calculating probability in scheme $\left[\textbf{A}\right]$.
    But its generalization can serve as a starting point for the emergence of new formulas
    for the calculation and geometric interpretation of probabilities in some urn schemes.

    In this note we consider some modifications of scheme $\left[\textbf{A}\right]$
    in which a sequence of urns can be involved, and we propose formulas that simplify the
    method of directly calculating probabilities without using elements of combinatorics.

\section{Modifications of scheme $\left[\textbf{A}\right]$}

    In this section, we consider some schemes that be generalization of scheme $\left[\textbf{A}\right]$ .

    First we need the following assertion.

\begin{lemma*}              \label{MyLem:1}
    The following equality is fair:
\begin{equation}           \label{2}
    {{a + b} \over a}{1 \over {\textsf{C}_{a + b}^k }} \sum
    \limits_{i = 1}^k {i\textsf{C}_a^i \textsf{C}_b^{k - i}} = k.
\end{equation}
\end{lemma*}

\begin{proof}
    Initially we are convincing that according to the properties of binomial coefficients, the following relations are true:
$$
    {k \over {a + b}}\textsf{C}_{a + b}^k = \textsf{C}_{a + b - 1}^{k - 1}
$$
    and
\begin{equation}           \label{3}
    \sum\limits_{i = 0}^k {\textsf{C}_a^i \textsf{C}_b^{k - i}} = \textsf{C}_{a + b}^k.
\end{equation}
    Using these relations, we have
$$
    {1 \over {{\displaystyle {k} \over \displaystyle {a + b}}\textsf{C}_{a + b}^k }}\sum\limits_{i = 1}^k {i\textsf{C}_a^i \textsf{C}_b^{k - i} }
    = {1 \over {\textsf{C}_{a + b - 1}^{k - 1} }}\sum\limits_{i = 1}^k {i\textsf{C}_a^i \textsf{C}_b^{k - i} }
    = {{\sum\nolimits_{i = 1}^k {i\textsf{C}_a^i \textsf{C}_b^{k - i}}}
    \over {\sum\nolimits_{i = 0}^{k - 1} {\textsf{C}_{a - 1}^i \textsf{C}_b^{k - i - 1}}}}\raise 1.8pt\hbox{.}
$$
    On the other hand
$$
    \sum\limits_{i = 1}^k {i\textsf{C}_a^i \textsf{C}_b^{k - i} }
    = a\sum\limits_{i = 1}^k {\textsf{C}_{a - 1}^{i - 1} \textsf{C}_b^{k - i} }
    = a\sum\limits_{i = 0}^{k - 1} {\textsf{C}_{a - 1}^i \textsf{C}_b^{k - i - 1}}.
$$
    Therefore
$$
    {1 \over {{\displaystyle{k} \over \displaystyle{a + b}}
    \textsf{C}_{a + b}^k }}\sum\limits_{i = 1}^k {i\textsf{C}_a^i \textsf{C}_b^{k - i}}=a.
$$
    The last equality is equivalent to \eqref{2}.
\end{proof}

    {\textbf{Scheme} $\left[\textbf{A}(k)\right]$}. Consider the urns with the compositions from
    scheme {$\left[\textbf{A}\right]$}. From the urn $\Pi _0$ randomly taken out $\xi =k$ number
    of balls, and they are transferred to urn $\Pi$ , here $k$ is any natural number is such that
    $k\leq a+b$. In this case, formula \eqref{1} is generalized in the following theorem. Denote
\begin{equation*}
    \alpha : = {c \over {c + d + k}} \quad {\textrm{and}} \quad \beta : = {{c + k} \over {c + d + k}}\raise 1.8pt\hbox{.}
\end{equation*}

\begin{theorem}              \label{MyTh:1}
    For scheme $\left[\textbf{A}(k)\right]$, the probability $\textsf{P}_\Pi \left({\left. A \right|\xi =k}\right)$
    that the ball, randomly taken out from urn $\Pi$, turns out to be white, is calculated by the following formula:
\begin{equation}           \label{4}
    \textsf{P}_\Pi \left({\left. A \right|\xi =k}\right) = \alpha + \left( \beta - \alpha \right)\theta,
\end{equation}
    where $\theta = \textsf{P}_{\Pi_0}(A)$ is the probability of the appearance of the white ball in the urn $\Pi _0$.
\end{theorem}

\begin{remark*}
    The expression on the right-hand side of formula \eqref{4} resembles intermediate points (values) in
    the Lagrange formula in mathematical analysis. It's obvious that
\begin{equation*}
    \alpha \leq \textsf{P}_\Pi \left({\left. A \right|\xi =k}\right) \leq \beta
\end{equation*}
    and the numbers $\alpha$ and $\beta$ are nothing more than the smallest and greatest values
    of all possible probabilities of the appearance of white ball in urn $\Pi$, provided that
    $\xi = k$ balls are randomly transferred from urn $\Pi _0$ to urn $\Pi$. It can be seen from
    the formula \eqref{4} that probability $\textsf{P}_\Pi \left({\left. A \right|\xi =k}\right)$
    is in such internal point of the ``value interval`` $[\alpha, \beta]$ that it divides
    this interval proportionally to $\theta$; see Figure~\ref{Figure_2} below.
\end{remark*}

\begin{proof}[Proof of Theorem~\ref{MyTh:1}]
    We first consider a case $\xi =2$. In this case by direct calculation we find
\begin{eqnarray*}
    {\textsf{P}_\Pi \left({\left. A \right|\xi =2}\right)}
    & = & {{(c + 2) \textsf{C}_a^2 + (c + 1) \textsf{C}_a^1 \textsf{C}_b^1 + c \textsf{C}_b^2} \over {(c + d + 2) \textsf{C}_{a + b}^2}} \\
\\
    & = & {{2a + (a + b)c} \over {(a + b)(c + d + 2)}} = {c \over {c + d + 2}} + {a \over {a + b}}{2 \over {c + d + 2}} \raise 1.8pt\hbox{.}
\end{eqnarray*}

    For arbitrary $\xi =k$, according to standard reasoning
\begin{equation}           \label{5}
    {\textsf{P}_\Pi \left({\left. A \right|\xi =2}\right)} = {1 \over {(c + d + k)
    \textsf{C}_{a + b}^k }}\sum\limits_{i = 0}^k {(c + i)\textsf{C}_a^i \textsf{C}_b^{k - i}}.
\end{equation} 
    Using formula \eqref{2} and equality \eqref{3} we transform
    the expression on the right-hand side of equality \eqref{5} to the form
\begin{eqnarray*}
    {1 \over {\textsf{C}_{a + b}^k }}\sum\limits_{i = 0}^k {(c + i)\textsf{C}_a^i \textsf{C}_b^{k - i}}
    & = & {c \over {\textsf{C}_{a + b}^k }}\sum\limits_{i = 0}^k {\textsf{C}_a^i \textsf{C}_b^{k - i}}
    + {1 \over {\textsf{C}_{a + b}^k }}\sum\limits_{i = 0}^k {i\textsf{C}_a^i \textsf{C}_b^{k - i} } \\
\\
    & = & c + {{ak} \over {a + b}} \raise 1.8pt\hbox{.}
\end{eqnarray*}
    Applying last equality in \eqref{5}, we have
\begin{equation}           \label{6}
    \textsf{P}_\Pi \left( {\left. A \right|\xi = k} \right)
    = {c \over {c + d + k}} + {a \over {a + b}}{k \over {c + d + k}}  \raise 1.8pt\hbox{.}
\end{equation} 
    Since
$$
    \textsf{P}_{\Pi _0}(A) = {a \over {a + b}}
$$
    according to our notation, the equality \eqref{6} is equivalent to \eqref{4}.

    Theorem~\ref{MyTh:1} is proved.
\end{proof}

    Continuing the discussion of Scheme~$\left[\textbf{A}(k)\right]$, we now set eyes on the geometric
    interpretation of Theorem~\ref{MyTh:1}. In the Figure~\ref{Figure_2}, in the orthogonal coordinate
    system  we put all possible values of probability $\theta = \textsf{P}_{\Pi _0}(A)$ along the
    vertical axis. On the horizontal axis we place the ``value interval`` $[\alpha, \beta]$ of
    probability $\textsf{P}_\Pi \left( {\left. A \right|\xi = k} \right)$. Draw the straight line
    $l_{\alpha}$ through points $\alpha$ and $\beta'$. Denote $\textsf{P}$ the intersection point
    of line $l_{\alpha}$ with the horizontal one $l_{\theta}$ passing through the point $\theta$.

    According to the similarity criterion of triangles, the following equalities are true:
$$
    {{\textsf{P}_\Pi - \alpha} \over {\beta - \alpha}}
    = {{\left| {\alpha \textsf{P}}\right|} \over {\left| {\alpha \beta'} \right|}}
    = {{\left| {{\textsf{P}_\Pi} \textsf{P}}\right|} \over {\left| {{\textsf{P}_\Pi}  {\textsf{P}'}} \right|}} = \theta.
$$
    This relation immediately implies formula \eqref{2}.
    The conclusion made corresponds with the geometric definition of probability, since
\begin{equation*}
    \textsf{P}_\Pi \left( {\left. A \right|\xi = k} \right)
    = {{measure \; of \; quadrangle \; S_\textsf{P}} \over {measure \; of \; unitary \; square \; S}}   \raise 1.8pt\hbox{.}
\end{equation*} 

\begin{figure}[htp]
\begin{center}
  \includegraphics[width=3.2in]{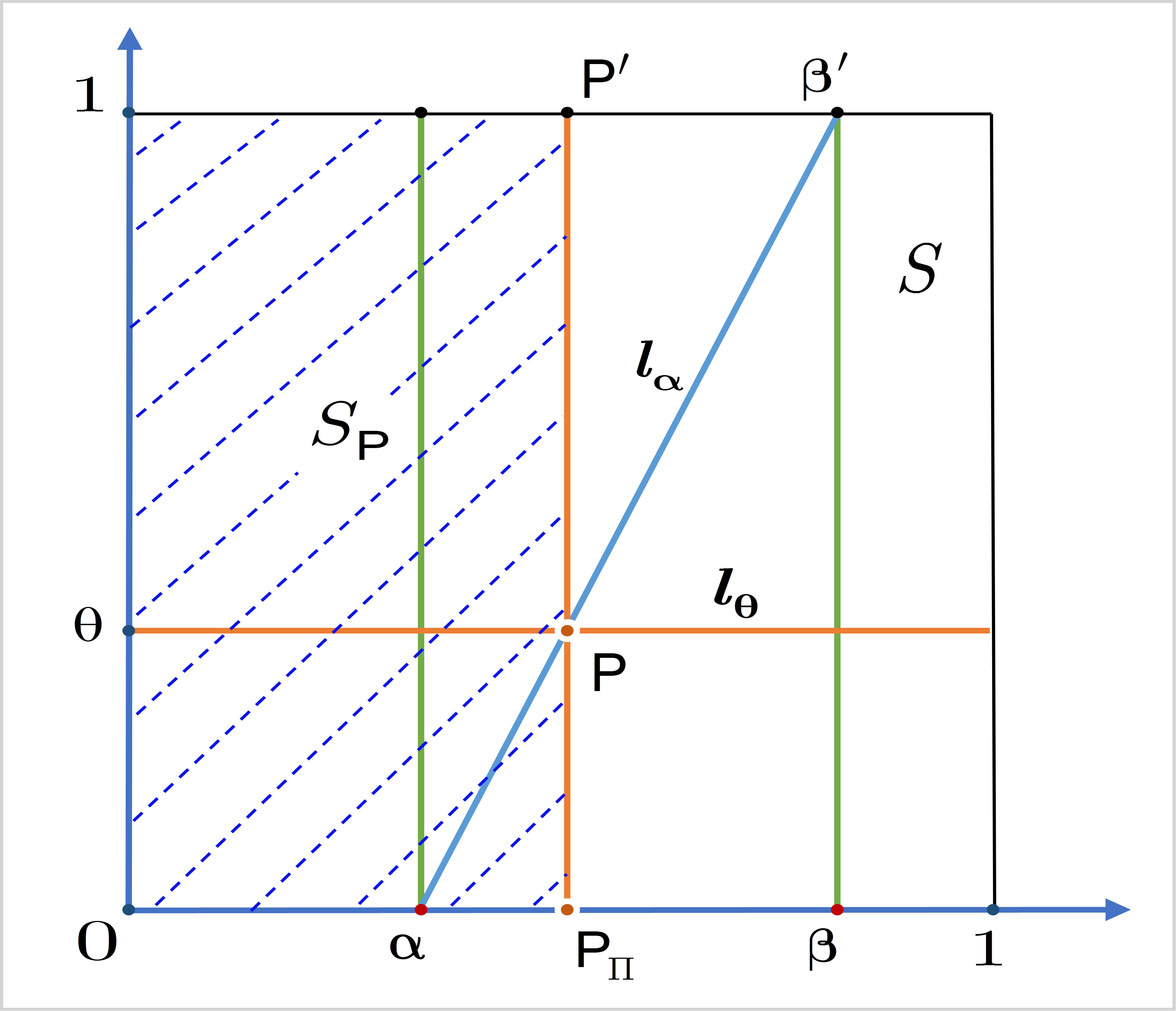}\\
  \caption{\textit{Geometric interpretation of} $\textsf{P}_\Pi \left({\left. A \right|\xi =k}\right)$} \label{Figure_2}
  \end{center}
\end{figure}

    Now in Scheme~$\left[\textbf{A}(k)\right]$ we introduce the following notation: let in urn $\Pi _0$ contains
    $M$ balls, of which $a$ are white, and in urn $\Pi$  there are $N$ balls, of which $c$ are white.

\begin{figure}[htp]
\begin{center}
  \includegraphics[width=3.2in]{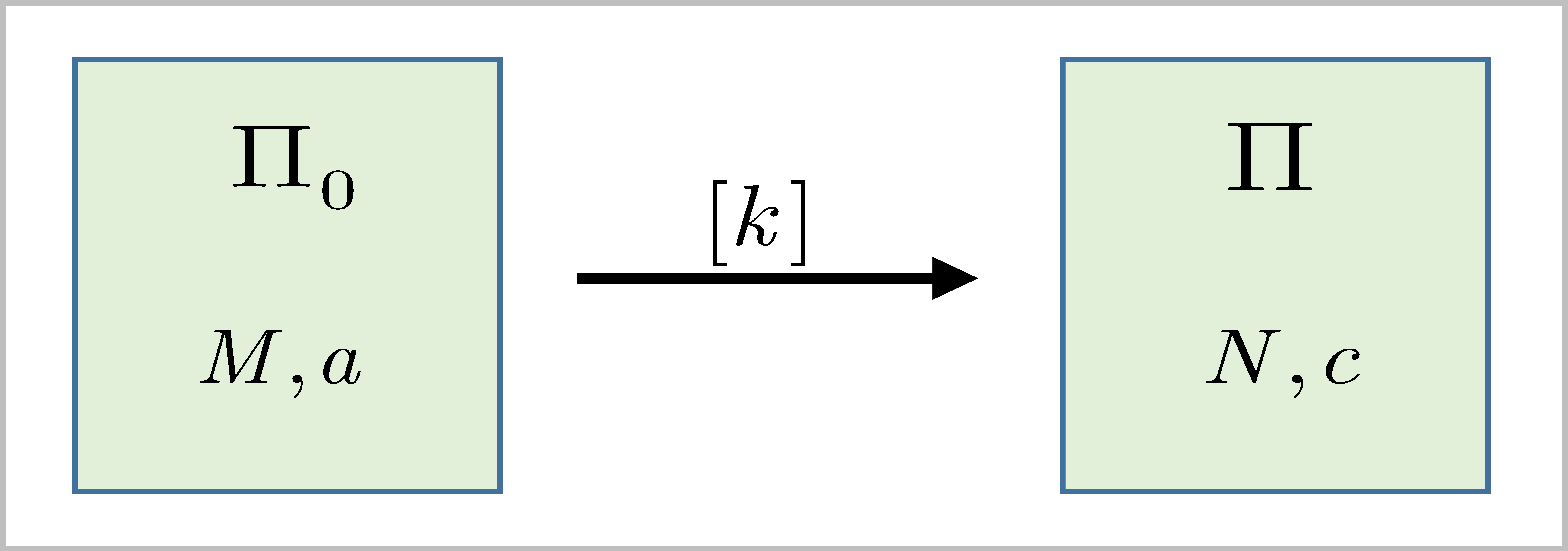}\\
  \caption{\textit{Scheme} $\left[\textbf{A}(k)\right]$} \label{Figure_3}
  \end{center}
\end{figure}

    In these designations Theorem~\ref{MyTh:1} can be reformulated as following.

\begin{theorem}              \label{MyTh:2}
    In scheme $\left[\textbf{A}(k)\right]$, the probability $\textsf{P}_\Pi \left({\left. A \right|\xi =k}\right)$
    that the ball, randomly taken out from urn $\Pi$, turns out to be white, is calculated by the following formula:
\begin{equation}           \label{7}
    \textsf{P}_\Pi \left({\left. A \right|\xi =k}\right) = {{\theta k + c} \over {N + k}}  \, \raise 1.8pt\hbox{,}
\end{equation}
    where $\theta = {a \mathord{\left/ {\vphantom {a M}} \right. \kern-\nulldelimiterspace} M}$.
\end{theorem}

\begin{proof}[Proof of Theorem~\ref{MyTh:2}]
    Formula~\eqref{7} follows from equality \eqref{6}.
\end{proof}

\begin{exercise}              \label{MyEx:1}
    There are 15000 details in first warehouse $\Pi_0$, from them 11850 are standard. The second warehouse
    $\Pi$ contain 17000 details, from them 15800 are standard. 10000 details were randomly taken from the
    warehouse $\Pi_0$ and transferred to the second one. Find the probability of the event that the
    detail taken randomly from the second warehouse $\Pi$, to be appeared standard.
\end{exercise}

\begin{proof}[\textbf{Solution}]
    We find the desired probability by the formula \eqref{7} as follows:
\begin{equation*}
    \textsf{P}_\Pi \left({\left. A \right|\xi =10000}\right)
    = {{{\displaystyle{11850} \over \displaystyle{15000}}10000 + 15800} \over {27000}} = {{79} \over {90}}
\end{equation*}
    or the same as $\textsf{P}_\Pi \left({\left. A \right|\xi =10000}\right) =  {{79} / {90}} \approx 0.878$.
\end{proof}

    {\textbf{Scheme} $\left[\textbf{A}_n(\mathbf{k})\right]$}. There are $n+1$ urns: $\Pi$ and $\Pi_1, \Pi_2, \ldots , \Pi_n$.
    In urn $\Pi_1$ there are $M_1$ balls, of which $a_1$ are white; in urn $\Pi_2$ there are $M_2$ balls, of which $a_2$ are white;
    and so forth, in urn $\Pi_n$ there are $M_n$ balls, of which $a_n$ are white. In urn $\Pi$ there are $N$ balls, of which $c$ are white.

\begin{figure}[htp]
\begin{center}
  \includegraphics[width=3.8in]{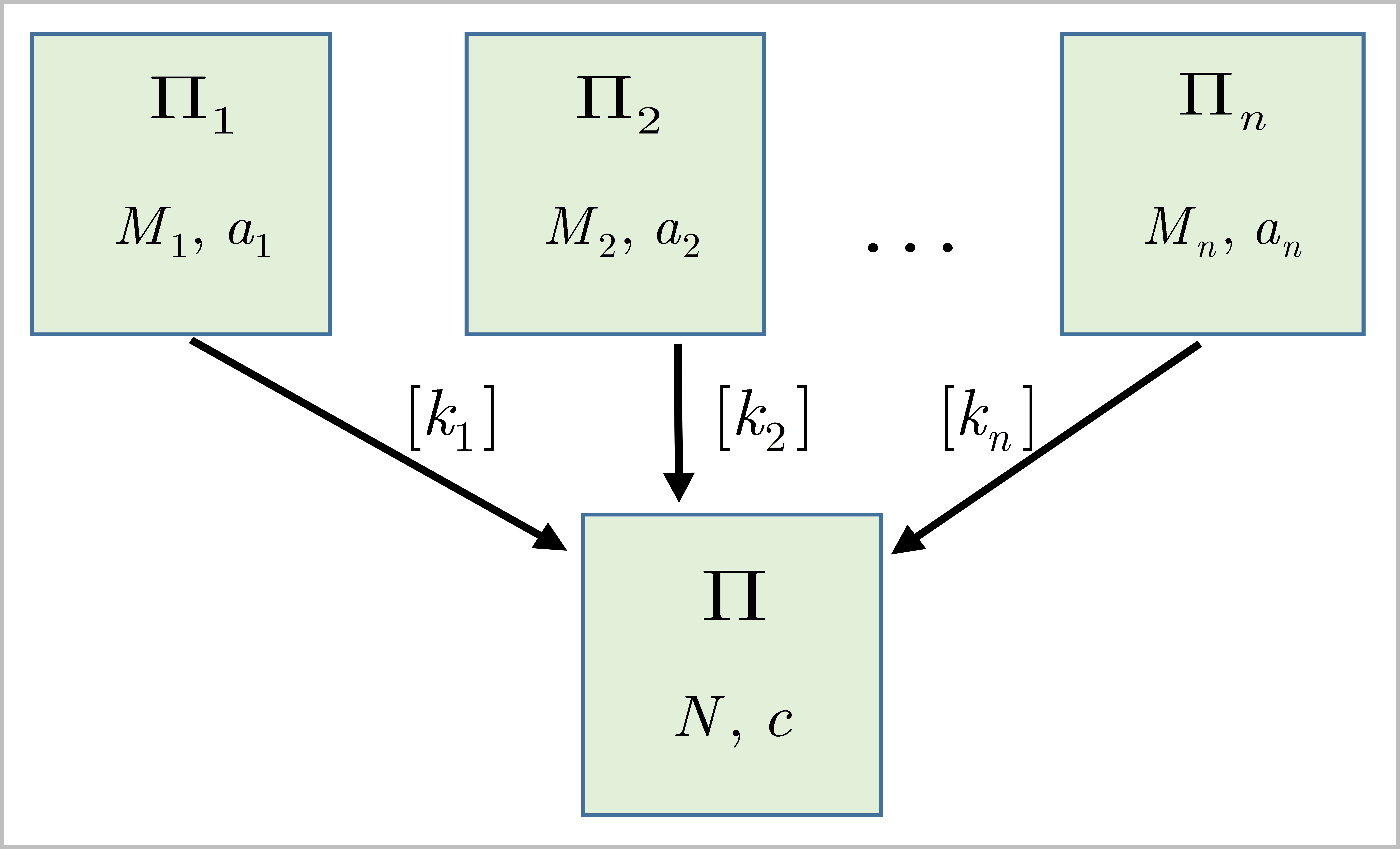}\\
  \caption{\textit{Scheme} $\left[\textbf{A}_n(\mathbf{k})\right]$} \label{Figure_4}
  \end{center}
\end{figure}

    From all urns $\Pi_1, \Pi_2, \, \ldots \, \Pi_n$ randomly selected $k_1, k_2, \, \ldots \,, k_n$
    balls respectively, and they are transferred to urn $\Pi$. Introduce the random vector
    $\mathbf{x} = \left(\xi_1, \xi_2, \ldots , \xi_n \right)$ characterizing the number of
    balls taken out of $n$ numbered urns and transferred to urn $\Pi$. The event
    $\{\mathbf{x} = \mathbf{k}\}$ means that the number of balls transferred to urn $\Pi$
    is equal to $\sum\nolimits_{i = 1}^n {k_i}$, the sum of the components of the vector
    $\mathbf{k} = \left(k_1, k_2, \ldots , k_n \right)$, where all $k_i\geq 0$. Scheme
    $\left[\textbf{A}_n(\mathbf{k})\right]$ is a natural generalization
    of scheme $\left[\textbf{A}(k)\right]$.

    In this case we have the following statement generalizing Theorem~\ref{MyTh:2}.

\begin{theorem}              \label{MyTh:3}
    In scheme $\left[\textbf{A}_n(\mathbf{k})\right]$, the probability
    $\textsf{P}_\Pi \left({\left. A \right|\mathbf{x} =\mathbf{k}}\right)$ that the ball, randomly
    taken out from urn $\Pi$, turns out to be white, is calculated by the following formula:
\begin{equation}           \label{8}
    \textsf{P}_\Pi \left({\left. A \right|\mathbf{x} =\mathbf{k}}\right)
    = {{\theta_{1} k_{1}+\theta_{2}k_{2} + \cdots + \theta_{n}k_{n} +c}
    \over {N + k_{1} +k_{2}+ \cdots +k_{n}}}  \, \raise 1.8pt\hbox{,}
\end{equation}
    where $\theta_{i} = {a_{i} \mathord{\left/ {\vphantom {a_{i} M_{i}}} \right. \kern-\nulldelimiterspace} M_{i}}$.
\end{theorem}

\begin{proof}
    The proof may be omitted  since it repeats the same arguments as in the proof of Theorem~\ref{MyTh:2}.
\end{proof}

\begin{exercise}              \label{MyEx:2}
    Three urns contained lots of balls of different colors. Of all balls in these urns there were:
    10 white from all 20 balls in urn $\Pi_1$, 15 white from all 25 balls in urn $\Pi_2$ and
    20 white from all 30 balls in urn $\Pi_3$. From these urns 5, 10 and 15 balls, respectively,
    were randomly taken out, and then they were transferred to the empty urn $\Pi$. Next, all balls
    in urn $\Pi$ randomly and in the appropriate amount were divided into three parts, and they all
    returned back to the primary urns. Eventually the amounts of the balls in numbered urns remained
    the same. We seek out the probability that the ball taken randomly from any urn will
    turn out to be white. For instance we will find the probability for urn $\Pi_2$.
\end{exercise}

\begin{proof}[\textbf{Solution}]
    First using formula~\eqref{8} for the case of $n=3$, $N=c=0$ and $\mathbf{k} = \left(5, 10, 15 \right)$
    we find the probability $\textsf{P}_\Pi \left({\left. A \right|\mathbf{x} =\mathbf{k}}\right)$ that
    the ball, randomly taken out from urn $\Pi$, turns out to be white:
\begin{equation*}
    \textsf{P}_\Pi \left({\left. A \right|\mathbf{x} =\mathbf{k}}\right)= {{{\displaystyle {10}
    \over \displaystyle {20}} 5 + {\displaystyle {15}
    \over \displaystyle {25}} 10 + {\displaystyle {20}
    \over \displaystyle {30}} 15} \over {5 + 10 + 15}} = {{37} \over {60}} \raise 1.8pt\hbox{.}
\end{equation*}
    To find the sought-for probability $\textsf{P}_{\Pi_2} \left({\left. A \right|\xi =10}\right)$, we
    just use formula~\eqref{7} for the cases of $N=15$, $k=10$ and $\theta = {37/60}$. Note that in
    the initial state $\textsf{P}_{\Pi_2} \left(A \right) = {15/25}$. After transfer $15$ balls are
    left in the urn $\Pi_2$ so that ${15/25} = {c/15}$. Hence $c= {{15} \over {25}}15$. Thus we have
\begin{equation*}
    \textsf{P}_{\Pi_2} \left({\left. A \right|\xi =10}\right) = {{{\displaystyle {37}
    \over \displaystyle {60}} 10 + {\displaystyle {15}
    \over \displaystyle {25}} 15} \over {15 + 10}} = {{182} \over {300}}
\end{equation*}
    or the same as $\textsf{P}_{\Pi_2} \left({\left. A \right|\xi =10}\right) =  {{182} / {300}} \approx 0.607$.
\end{proof}

    {\textbf{Scheme}} $\left[\textbf{A}_n\left(\textbf{\textit{a}}_s, \mathbf{K}\right)\right]$.
    There is a sequence of urns $\left\{{\Pi _i ,i = 0,1,2, \ldots } \right\}$. Each urn contains
    particles of $s$ types $T_1$, $T_2$, \ldots , $T_s$. The state of the $i$th urn is determined
    by the $s$ dimensional vector $\textbf{\textit{a}}_s^i=\left({a_1^i, a_2^i, \ldots , a_s^i} \right)$,
    where $a_j^i$ is the number of particles of the type of $T_j$ in the $i$th urn. We observe the
    operation $\textbf{K}: = \left[ {k_0, k_1, \ldots} \right]$ which is as follows: $k_0$ particles
    are randomly taken out from urn $\Pi _0$ and transferred to urn $\Pi _1$; after $k_1$ particles
    are randomly taken out from urn $\Pi _1$ and transferred to urn $\Pi _2$ and so forth.  We use
    the symbol $\textbf{K}_m: = \left[ {k_0, k_1, \ldots , k_{m-1}} \right]$ in the case when the
    operation $\textbf{K}$ is observed in the $m$th step, where $m=1,2, \ldots$.

\begin{figure}[htp]
\begin{center}
  \includegraphics[width=5.4in]{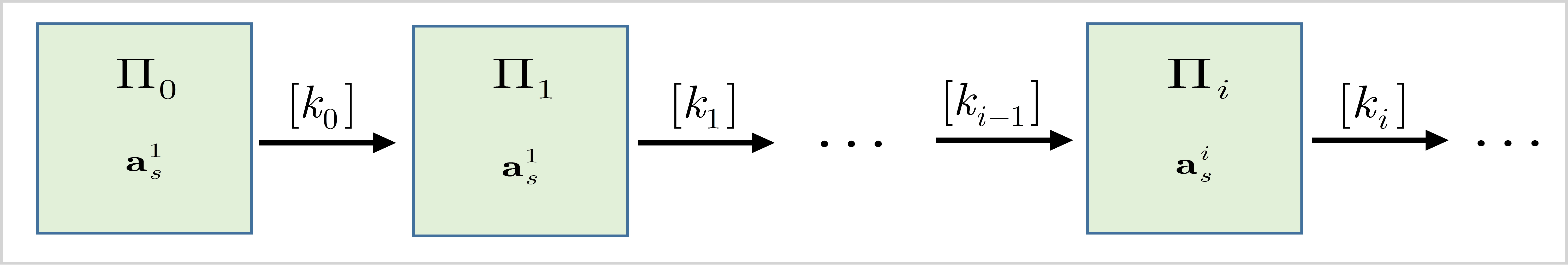}\\
  \caption{\textit{Scheme} $\left[\textbf{A}_n\left(\textbf{\textit{a}}_s, \mathbf{K}\right)\right]$} \label{Figure_5}
  \end{center}
\end{figure}

    Fix the type
    $T_\textsf{s}$, $1 \leq \textsf{s} \leq s$. Denote by $\textsf{P}_{\Pi_1} \left({\left. T_\textsf{s}
    \right|\textbf{K}_1}\right)$ the probability that the particle randomly taken from urn $\Pi_{1}$
    provided the operation $\textbf{K}_1$, will turn out to be a particle of the
    type $T_\textsf{s}$. According to Theorem~\ref{MyTh:2}, we have
\begin{equation}           \label{9}
    \textsf{P}_{\Pi_1} \left({\left. T_\textsf{s}\right|\textbf{K}_1}\right)
    = \alpha_{\textsf{s}1} + \left(\beta_{\textsf{s}1} - \alpha_{\textsf{s}1} \right)\theta_{\textsf{s}} ,
\end{equation}
    where $\theta_{\textsf{s}}  = \textsf{P}_{\Pi _0} \left( {T_\textsf{s} } \right)$ is the probability
    of an appearance of the particle of type $T_\textsf{s}$ in urn $\Pi _0$ and
\begin{equation*}
    \alpha_{\textsf{s}1} = {{a_\textsf{s}^1} \over {\sum\nolimits_{j = 1}^s {a_j^1}+ k_0}} \quad {\textrm{and}} \quad
    \beta_{\textsf{s}1} = {{a_\textsf{s}^1 + k_0} \over {\sum\nolimits_{j = 1}^s {a_j^1}+ k_0}}  \raise 1.8pt\hbox{.}
\end{equation*}

    Continuing discussions, we can generalize the formula \eqref{9} for all $\textsf{P}_{\Pi_m} \left({\left. T_\textsf{s}
    \right|\textbf{K}_m}\right)$, $m=2,3, \ldots $, and obtain a recurrence formula that would allows us to find the
    probability of appearance of particle of the type $T_\textsf{s}$ in arbitrary urn at the any step of operation
    $\textbf{K}$. We have the following theorem.

\begin{theorem}              \label{MyTh:4}
    In scheme $\left[\textbf{A}_n\left(\textbf{\textit{a}}_s, \mathbf{K}\right)\right]$, the probability
    $\textsf{P}_{\Pi_m} \left({\left. T_\textsf{s} \right|\textbf{K}_m}\right)$ that the particle,
    randomly taken out from urn $\Pi_m$, provided that operation $\textbf{K}_m$, turns out to be the
    particle of type $T_\textsf{s}$, is calculated by the following formula:
\begin{equation}           \label{10}
    \textsf{P}_{\Pi_m} \left({\left. T_\textsf{s} \right|\textbf{K}_m}\right)
    = \alpha_{\textsf{s}1} + \left(\beta_{\textsf{s}1} - \alpha_{\textsf{s}1} \right)
    \textsf{P}_{\Pi_{m-1}} \left({\left. T_\textsf{s} \right|\textbf{K}_{m-1}}\right),
\end{equation}
    where $m=2,3, \ldots$
    \begin{equation*}
    \alpha_{\textsf{s}1} = {{a_\textsf{s}^m} \over {\sum\nolimits_{j = 1}^s {a_j^m}+ k_{m-1}}} \quad {\textrm{and}} \quad
    \beta_{\textsf{s}1} = {{a_\textsf{s}^m + k_{m-1}} \over {\sum\nolimits_{j = 1}^s {a_j^m}+ k_{m-1}}}  \raise 1.8pt\hbox{,}
\end{equation*}
     here $\textsf{P}_{\Pi_1} \left({\left. T_\textsf{s}\right|\textbf{K}_1}\right)$ is in \eqref{9}.
\end{theorem}

\begin{exercise}              \label{MyEx:3}
    Three urns contain balls of white, black and yellow colors. Urn $\Pi_0$ contains 310 white,
    350 black and 370 yellow balls; urn $\Pi_1$ contains no white ball, 480 black and 530 yellow
    balls; urn $\Pi_2$ contains 600 white, 640 black and 670 yellow balls. 500 balls are randomly
    taken from urn $\Pi_0$ and they transferred to the urn $\Pi_1$ and then 600 balls are randomly
    taken from urn $\Pi_1$ and transferred to the urn $\Pi_2$; see Figure~\ref{Figure_6} below.
    We are looking for the probabilities that a ball randomly selected from urn $\Pi_2$ will
    turn out to be white, black, and yellow, respectively.
\end{exercise}

\begin{figure}[htp]
\begin{center}
  \includegraphics[width=3.8in]{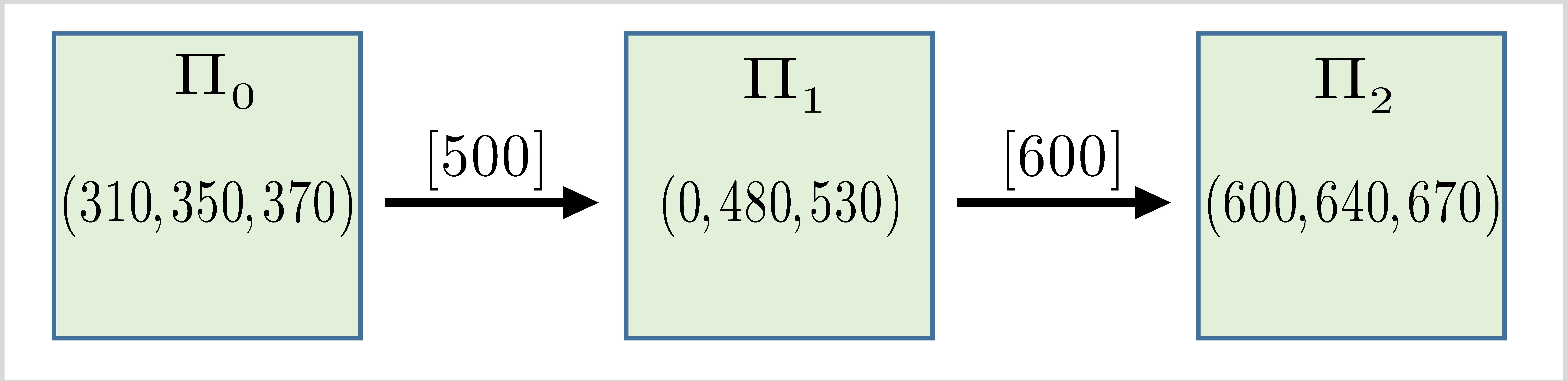}\\
  \caption{\textit{Illustration of Exercise}~\ref{MyEx:3}} \label{Figure_6}
  \end{center}
\end{figure}

\begin{proof}[\textbf{Solution}]
    Let $T_\textsf{w}$ be white, $T_\textsf{b}$ -- black and $T_\textsf{y}$ -- yellow types. It's obvious that
\begin{equation*}
    \textsf{P}_{\Pi_0} \left({T_\textsf{w}}\right) = {{310} \over {1030}}
    \quad {\textrm{and}} \quad
    \textsf{P}_{\Pi_0} \left({T_\textsf{b}}\right) = {{350} \over {1030}}
    \quad {\textrm{and}} \quad
    \textsf{P}_{\Pi_0} \left({T_\textsf{y}}\right) = {{370} \over {1030}} \raise 1.8pt\hbox{.}
\end{equation*}
    We can easily calculate, that
\begin{equation*}
    \alpha_{\textsf{w}1} = {{0} \over {1510}} \quad {\textrm{and}} \quad
    \beta_{\textsf{w}1} =  {{500} \over {1510}} ;
\end{equation*}
    and
\begin{equation*}
    \alpha_{\textsf{b}1} = {{480} \over {1510}} \quad {\textrm{and}} \quad
    \beta_{\textsf{b}1} =  {{980} \over {1510}} ;
\end{equation*}
    and
\begin{equation*}
    \alpha_{\textsf{y}1} = {{530} \over {1510}} \quad {\textrm{and}} \quad
    \beta_{\textsf{y}1} =  {{1030} \over {1510}} \raise 1.8pt\hbox{.}
\end{equation*}
    Using the data above and formula~\eqref{9} we obtain
\begin{equation*}
    \textsf{P}_{\Pi_1} \left({\left. T_\textsf{w}\right|\textbf{K}_1}\right)
    = 0 + \left({{500} \over {1510}} - 0\right){{310} \over {1030}} = {{1550} \over {15553}} \, \raise 1.8pt\hbox{,}
\end{equation*}

\begin{equation*}
    \textsf{P}_{\Pi_1} \left({\left. T_\textsf{b}\right|\textbf{K}_1}\right)
    = {{480} \over {1510}} + \left({{980} \over {1510}} - {{480} \over {1510}}\right){{350} \over {1030}}
    = {{6694} \over {15553}} \, \raise 1.8pt\hbox{,}
\end{equation*}

\begin{equation*}
    \textsf{P}_{\Pi_1} \left({\left. T_\textsf{y}\right|\textbf{K}_1}\right)
    = {{530} \over {1510}} + \left({{1030} \over {1510}} - {{530} \over {1510}}\right){{370} \over {1030}}
    = {{7309} \over {15553}} \, \raise 1.8pt\hbox{,}
\end{equation*}

    By the same way and using the formula~\eqref{10} we can find sought probabilities as follows:
\begin{equation*}
    \textsf{P}_{\Pi_2} \left({\left. T_\textsf{w}\right|\textbf{K}_2}\right)
    = {{600} \over {2510}} + \left({{1200} \over {2510}} - {{600} \over {2510}}\right){{1550} \over {15553}}
    = {{1026180} \over {3903803}} \left(\approx 0.263\right),
\end{equation*}

\begin{equation*}
    \textsf{P}_{\Pi_2} \left({\left. T_\textsf{b}\right|\textbf{K}_2}\right)
    = {{640} \over {2510}} + \left({{1240} \over {2510}} - {{640} \over {2510}}\right){{6694} \over {15553}}
    = {{1397032} \over {3903803}} \left(\approx 0.358\right),
\end{equation*}

\begin{equation*}
    \textsf{P}_{\Pi_2} \left({\left. T_\textsf{y}\right|\textbf{K}_2}\right)
    = {{670} \over {2510}} + \left({{1270} \over {2510}} - {{670} \over {2510}}\right){{7309} \over {15553}}
    = {{1480591} \over {3903803}} \left(\approx 0.379\right).
\end{equation*}
    It can be checked that $\textsf{P}_{\Pi_2} \left({\left. T_\textsf{w}\right|\textbf{K}_2}\right)+
    \textsf{P}_{\Pi_2} \left({\left. T_\textsf{b}\right|\textbf{K}_2}\right)+
    \textsf{P}_{\Pi_2} \left({\left. T_\textsf{y}\right|\textbf{K}_2}\right)=1$.
\end{proof}

\section*{Acknowledgements}

   The basic presuppositions of the formulas proposed in the paper was initiated due to
   Yu.Khodjaev -- the student of Karshi State University. Dr.~A.Imomov, being his scientific
   supervisor, mathematically systematized and stood on the whole process of proofs of formulas.


\end{document}